\documentclass[11pt,dvipsnames,reqno]{amsart}
\usepackage{times}
\usepackage[all]{xy}
\usepackage{amssymb,amsmath,amsthm}
\newcommand{\calL}{\mathcal{L}}
\newcommand{\calO}{\mathcal{O}}

\newcommand{\calT}{\mathcal{T}}

\newcommand{\calA}{\mathcal{A}}
\newcommand{\calB}{\mathcal{B}}

\newcommand{\bbP}{\mathbb{P}}
\newcommand{\bbZ}{\mathbb{Z}}

\newcommand{\bbG}{\mathbb{G}}
\newcommand{\bbT}{\mathbb{T}}
\newcommand{\bbC}{\mathbb{C}}

\newcommand{\bbF}{\mathbb{F}}

\newcommand{\frakC}{\mathfrak{C}}
\newcommand{\Foc}{\mathrm{Foc}}

\usepackage[
            pdftex,
           colorlinks=true,
            urlcolor=blue,       
            filecolor=blue,     
            linkcolor=blue,       
            citecolor=blue,         
            pdftitle= {Title},
            pdfauthor={Author},
            pdfsubject={Subject},
            pdfkeywords={Key}
            pagebackref,
            pdfpagemode=None,
            bookmarksopen=true]
            {hyperref}

\newcommand{\frakS}{\mathfrak{S}}

\newtheorem{thm}{Theorem}[section]
 \newtheorem{prop}[thm]{Proposition}
 \newtheorem{lem}[thm]{Lemma}

 \theoremstyle{definition}
 \newtheorem{example}[thm]{Example}
 
 \newtheorem{remark}[thm]{Remark}

\newcommand{\Kum}{\mathrm{Kum}}
\newcommand{\Jac}{\mathrm{Jac}}

\newcommand{\la}{\langle}
\newcommand{\ra}{\rangle}

\newcommand{\beq}{\begin{equation}}
\newcommand{\eeq}{\end{equation}}

\author{Igor V. Dolgachev}
\address{\hfill \newline 
Department of Mathematics \newline
University of Michigan \newline
525 East University Avenue \newline
Ann Arbor,
 MI 48109-1109 USA}
\email{idolga@umich.edu}

\begin{document}
\title[K3 surfaces of Kummer type]
{K3 surfaces of Kummer type in characteristic two}

\date{}
\dedicatory{\hspace{5cm}
We will go the other way.\\
\vspace{3pt}\hspace{7cm} Vladimir Illych Lenin}

\begin{abstract} We discuss K3 surfaces in characteristic two 
that contain the Kummer configuration of smooth rational curves.
\end{abstract}

\maketitle

\section*{Introduction} 
The Kummer surface $\Kum(A)$ of an abelian surface $A$ over an algebraically closed field of characteristic $p$ 
is defined to be the quotient of $A$ by the negation 
involution $\iota:a\mapsto -a$. If $p\ne 2$, the abelian surface $A$ has $2^4$ two-torsion points that give 
rise to sixteen ordinary double points 
on $\Kum(A)$. A minimal resolution of singularities $X$ of $\Kum(A)$ is a K3 surface containing a set $\calA$ of 16 disjoint 
smooth rational curves on it ($(-2)$-curves for short because their self-intersection is equal to $-2$). Conversely, if 
$k = \bbC$, the field of complex numbers, a theorem of Nikulin asserts that a K3 surface containing 
a set of 16 disjoint $(-2)$-curves arises in this way from
the Kummer surface of some complex abelian surface.\footnote{This fact usually assumes that the ground field is the field 
of complex numbers, but, as shown to me by the anonymous referee, it is true if $p\ne 2$.}

Let $A$ be a simple principally polarized abelian surface, 
hence isomorphic to the Jacobian variety $\Jac(C)$ of a curve of genus 
$2$. The embedding of $C$ into $\Jac(C)$ can be chosen in such a way that its image $\Theta$ is invariant under the involution 
$\iota$. The linear system $|2\Theta|$ defines 
a regular map $\phi:A\to |2\Theta|^* \cong \bbP^3$ that factors through $\Kum(A)$. This map embeds $\Kum(A)$ into $\bbP^3$ as a quartic 
surface with 16 ordinary double points. These surfaces have been studied for two hundred years; we refer 
to \cite{DolgachevKum} for the history.
The restriction of the map $\phi$ to a translate of $\Theta$ by a $2$-torsion point is a degree two map to a conic in 
$\Kum(A)$ ramified over 6 points. The set $\calB$ of proper transforms of the sixteen conics in $X$ 
also consists of disjoint $(-2)$-curves. The incidence relation between the two sets defines an abstract symmetric configuration 
$(16_6)$, the \emph{Kummer configuration}. 

If $A$ is a non-simple abelian surface, i.e. $A$ is isomorphic to the product 
$E_1\times E_2$ of elliptic curves, the symmetric principal polarization $\Theta$ can be chosen to be 
equal to $(E_1\times \{0\})\cup (\{0\}\times E_2)$. The map $\phi:A\to \bbP^3$ defined by $|2\Theta|$ is of degree 
$4$ onto a smooth quadric $Q$. The union of the images of the translates of $\Theta$ is the union of 8 lines 
$L_i,M_i$ on $Q$, 
four from each of the two
rulings. The double cover $X'$ of $Q$ branched along these eight lines has $16$ ordinary points, and it is birationally 
isomorphic to a K3 surface $X$. The surface contains a set $\calA$ of disjoint sixteen $(-2)$-curves equal to the exceptional curves 
$E_{ij},1\le i,j\le 4,$  of a minimal resolution 
of singularities of $X'$. Another set of sixteen disjoint $(-2)$-curves consists of reducible $(-2)$-curves 
$\bar{L}_i+E_{ij}+\bar{M}_{j}$, where $\bar{L}_i,\bar{M_j}$ are reduced pre-images of the lines $L_i, M_j$. The two sets $(\calA,\calB)$ form the Kummer 
configuration $(16_6)$.

A beautiful aspect of the geometry of the Kummer surfaces of Jacobians of curves of genus 2 
is their relationship with the classical geometry of quadratic line complexes \cite[10.3]{CAG}. 
A Kummer surface appears as the \emph{singular surface} 
of a quadratic line complex $\frakC$, the locus of points $x\in \bbP^3$ such that the 
plane $\Omega(x)$ of lines containing $x$ intersects $\frakC$ along a singular conic. The set of irreducible components 
of these conics (which are lines in $\frakC$) is isomorphic to the Jacobian variety of a curve $C$ of genus 2. 
The curve $C$ is isomorphic to the double cover 
of the pencil of quadrics containing $\frakC$ ramified over the set of six singular quadrics. The set of singular points of 
$\Omega(x)\cap \frakC, x\in \bbP^3,$ is an octic surface in the Pl\"ucker space $\bbP^5$ birationally isomorphic 
to the Kummer surface. It is singular in characteristic $2$  \cite{KatsuraKondo}.
 
A less-known construction, due to Kummer himself, relates the Kummer surface $\Kum(\Jac(C))$ to the theory
of congruences of lines in $\bbP^3$, irreducible surfaces in the Grassmannian $G_1(\bbP^3)$. The Kummer surface appears as the 
focal surface 
of a smooth congruence of lines $S$ of order 2 and class 2. The congruence $S$ is a quartic del 
Pezzo surface anti-canonically embedded in a hyperplane in the Pl\"ucker space $\bbP^5$. Its realization 
as a congruence of lines chooses a smooth anti-bicanonical curve $B\in |-2K_S|$ that touches all 16 lines on 
$S$. The double cover $X$ of $S$ branched along $B$ is a K3 surface birationally isomorphic to the Kummer surface 
$\Kum(\Jac(C))$ for some genus $2$ curve $C$. The $16$ lines on $S$ split into the union of two sets $\calA, \calB$ 
of disjoint $(-2)$-curves, 
which form the Kummer configuration. 

Let us see what is going wrong if we assume that $p = 2$. First of all, there are no normal quartic surfaces with 
16 nodes \cite{Catanese}. 
An abelian surface $A$ has four, two, or one $2$-torsion points 
depending on its $p$-rank $r$ equal to $2,1,0$, respectively. If $r = 2$ (resp. $r = 1$, resp. $r = 0$), the singular points of 
$\Kum(A):=A/(\iota)$ are four rational double points 
of type $D_4$ (resp. two rational double point of type $D_8$, resp. one elliptic double point) \cite{Katsura}. In the fist two cases, the Kummer surface is birationally isomorphic to a K3 surface, in the third case, it is a rational surface.
The linear system $|2\Theta|$ still defines a degree two map onto a quartic surface in $\bbP^3$. The equations 
of these surfaces can be found in \cite{Laszlo} if $r = 2$ and in \cite{Duquesne} for arbitrary $2$-rank. 

The relationship with the quadratic line complexes is studied in a recent paper of 
T. Katsura and S. Kond\={o} \cite{KatsuraKondo}. In characteristic 
$2$, a pencil of quadrics in $\bbP^5$ with smooth base locus $Y$ has three (instead of six) singular 
quadrics. 
The variety of lines in $Y$ is isomorphic to the Jacobian variety of a genus 2 curve with an Artin-Schreier cover 
of $\bbP^1$ of the form $y^2+a_3(t_0,t_1)y+a_6(t_0,t_1) = 0$, where the zeros of the binary cubic $a_3$ correspond to singular quadrics in the pencil
\cite{Bhosle}. Identifying one of the smooth members of the pencil with the 
Grassmannian $G_1(\bbP^3)$, one can consider, as in the 
case $p\ne 2$, the base locus of the pencil as a quadratic line complex $\frakC$.
The singular surface of the quadratic line complex is a Kummer 
quartic surface and the surface
in $\bbP^5$ of singular points of the conics $\Omega(x)\cap Y$ is an octic surface with 12 nodes birationally isomorphic 
to the Kummer surface. The equations of the quartic and the octic surfaces are provided in loc. cit..

The main drawback of this nice extension of the theory of Kummer surfaces to characteristic 2 is that 
the Kummer configuration and the relationship 
between 6 points in $\bbP^1$ gets lost. In the present paper, we will present another approach whose goal 
is to reconstruct 
these relationships. Although we lose the relationship to curves of genus two, 
we will restore the relationships with the
Kummer configuration $(16_6)$, sets of six points in $\bbP^1$, and 
the theory of congruences of lines in $\bbP^3$. The situation is very similar to what happens with del Pezzo surfaces of degree two (resp one). 
The  Geiser (resp. Bertini) involution  defines a separable Artin-Schreier double cover whose branch curve is a smooth 
conic
(resp. a rational quartic curve) instead of a plane quartic curve (resp. a canonical genus 4 curve on a singular quadric). The 
connection to these curves is lost, but their attributes such as 28 bitangents (resp. 120 tritangent planes) survive (see \cite{DM}).

The paper should be considered as a footnote to \cite{KatsuraKondo}. I am thankful to the authors for a helpful 
discussion. I am also 
grateful to the anonymous referee for many helpful comments and for detecting some computational errors.

\section{K3 surfaces of Kummer type} Let $k$ be an algebraically closed field of characteristic $p \ge 0$. 
We define a \emph{K3 surface of Kummer type} to be a K3 surface $X$ that contains two sets $\calA$ and $\calB$ of sixteen
 disjoint $(-2)$-curves (or their degenerations), such that any $A\in A$ intersects $n$ curves from $\calB$, and vice versa, every 
curve $B$ from $\calB$ intersects $n$ curves $A$ from $\calA$. In other words, the two sets $(\calA,\calB)$ form a 
symmetric abstract configuration 
$(16_n)$. We call the number $n$ the \emph{index} of $X$.

A classical example of a K3 surface of Kummer type of index $6$ is a minimal smooth model of the Kummer surface $\Kum(A)$ of a 
principally polarized abelian surface in characteristic $p\ne 2$.
 
 As we discussed in the introduction, in characteristic 2, the Kummer surfaces are still defined but 
they are not of Kummer type. We also explained how the 
 geometry of the Kummer surface of a principally polarized abelian surface $A$ in characteristic 
$p\ne 2$ is  related to the geometry 
of the sets of six points in $\bbP^1$. 
Namely, the double cover of $\bbP^1$ ramified over a set of six points is a smooth genus two 
curve $C$, and one can associate to $C$  
the Kummer surface $\Kum(\Jac(C))$ of the Jacobian variety $\Jac(C)$. 
When $A$ is not a simple abelian surface but rather the product 
$E_1\times E_2$ of two elliptic curves, we replace six points on $\bbP^1$ with  six points on 
a stable rational curve $C$ consisting of two irreducible components with 3 points on each component. 
The double cover of $C$ 
of degree 2 ramified over 6 points (and the intersection point of the components) is 
isomorphic to the union of two elliptic curves $E_1$ and $E_2$ intersecting at one point. Its 
generalized Jacobian variety is isomorphic to $E_1\times E_2$.  

\vskip3pt
  The following example of a K3 surface of Kummer type of index $10$ is less known.
 
 \begin{example} A \emph{Traynard surface} is a quartic surface in $\bbP^3$ over an 
 algebraically closed field $k$ of characteristic $p\ne 2$
 with two sets of disjoint lines 
 $\calA$ and $\calB$ that form a symmetric configuration $(16_{10})$. 
  These surfaces were constructed by Traynard \cite{Traynard} (see \cite{Godeaux} where the surfaces are named after Traynard). 
  Not being aware of Traynard's work, W. Barth and I. Nieto rediscovered the Traynard surfaces in 
  \cite{Barth}. The surfaces are embedded Kummer surfaces of simple abelian surfaces $A$ with polarization of type 
  $(1,3)$. The negation involution acts on the linear space $H^0(A,\calO_A(2\Theta))$, where $\Theta$ is a symmetric 
  polarization divisor.
  The eigensubspace $V$ with eigenvalue equal to $-1$ is of dimension 4. The linear 
  system $|V|\subset 
  |2\Theta|$ has base points at all $2$-torsion points of $A$ and 
  defines a finite map of degree 2 of the blow-up of these points to $\bbP^3$ with image a smooth 
  quartic surface $X$.
 The images of the exceptional curves over the torsion points form a set $\calA$ of 16 lines on $X$.
 The unique symmetric theta divisor  $\Theta$ is a curve of genus $4$, it passes through $10$ torsion points, 
 and the images of the translates of $\Theta$ by $2$-torsion points provides another set 
 $\calB$ of 16 disjoint lines on $X$.
 \end{example} 
  
  The next proposition is due to N. Shepherd-Barron \cite[Corollary 13]{S-B}:
  
  \begin{prop}\label{nick} Let $X$ be a K3 surface over a field of characteristic $2$. Suppose that $X$ contains 
  $\ge 13$ disjoint $(-2)$-curves. 
  Then it is unirational, and, in particular, a supersingular surface.
  \end{prop} 
  
  The following is an example of a K3 surface of Kummer type of index $4$ in characteristic $2$ \cite{DolgKondo2}.
  
  \begin{example}\label{DK} Let  $X$ be a supersingular K3 surface with the Artin invariant $\sigma_0$ equal to $1$. 
  The isomorphism class of $X$ is unique. The surface contains a quasi-elliptic 
  pencil with $5$ reducible fibers of type 
  $\tilde{D}_4$ and $16$ disjoint sections. The union of non-multiple irreducible components of 
  four reducible fibers is a set $\calA$ of $16$ disjoint $(-2)$-curves. 
  Another set is formed by the $16$ sections. 
  Each section intersects one non-multiple component in each fiber, and this easily gives that the sets 
  $\calA,\calB$ form a symmetric configuration of type $(16_4)$. So the surface is of K3 type of index $4$ in five different ways.

  \end{example}
  In the next sections, we give three different constructions of a $3$-dimensional familiy of supersingular K3 surfaces of 
  Kummer type and index $6$ in characteristic $2$. Its general member is a supersingular K3 surface with Artin 
  invariant $\sigma_0$ equal to $4$.

\section{Weddle surfaces}
There is an explicit relationship between sets of six points in $\bbP^1$ and Kummer surfaces. 
One uses the Veronese map to 
put the six points $p_1,\ldots,p_6$ on a twisted cubic $R_3$ in $\bbP^3$. 
The discriminant surface of the web $L$ of quadric surfaces through this set of 
six points is isomorphic (if $p\ne 2$) to $\Kum(\Jac(C))$. The curve $C$ is isomorphic to the 
double cover of $R_3$ branched over the six points $p_1,\ldots,p_6$.

In the case $p\ne 2$, the \emph{Weddle surface} $W$ is defined to be the locus of singular points of 
 quadrics from the web $L$. Equivalently, it can be defined as the closure of the locus of points 
 $x\in \bbP^3$ such the projections of the points $p_1,\ldots,p_6$ from $x$ lie on a conic.
 
We may choose the projective coordinates to assume that
\begin{eqnarray}\label{six}
\begin{split}
&p_1 = [1,0,0,0],\ p_2 = [0,1,0,0],\ p_3 = [0,0,1,0],\\
&p_4 = [0,0,0,1],\  p_5 = [1,1,1,1],\ 
p_6 = [a,b,c,d],
 \end{split}
\end{eqnarray}
where the point $p_6$ does not lie in any plane spanned by three of the points $p_i, i< 6$.   
Then the equation of the Weddle surface is 
\beq\label{eqn:Hutchinson}
 \det\begin{pmatrix}ayzw&x&1&a\\
bxzw&y&1&b\\
cxyw&z&1&c\\
dxyz&w&1&d\end{pmatrix} = 0\eeq
\cite[\S 97]{Hudson}.
One checks that, in all characteristics, a quartic surface $W$ given by this equation has ordinary double points 
$p_1,\ldots,p_6$. The surface $W$ contains the lines $\la p_i,p_j\ra$ and the twisted cubic $R_3$ 
through $p_1,\ldots,p_6$, all with multiplicity 1.

Conversely, counting parameters, we obtain that a general quartic surface in $\bbP^3$ containing six lines 
$\la p_i,p_j\ra$ and the twisted cubic 
$R_3$ passing through the points $p_1,\ldots,p_6$ is given by \eqref{eqn:Hutchinson} 

If $p = 2$, the symmetric matrix of the polar bilinear form of quadrics from the web $L$
is an alternating form, so the discriminant surface is given by the 
pfaffian, hence it is a quadric surface. However, a quadric with polar bilinear form of corank $2$ can still
have an isolated singular point. In fact, one checks that the geometric description of the Weddle surface still holds in characteristic $2$
 
Another peculiarity of the case $ p = 2$ is that $W$ has an additional singular point 
$$P:= [\sqrt{a},\sqrt{b},\sqrt{c},\sqrt{d}].$$
A direct computation of the resolution of this singular point shows that it is rational double point of type $D_4$.

\begin{prop} A minimal nonsingular model of the Weddle surface is a supersingular K3 surface of Kummer type and index 6.
\end{prop} 

\begin{proof} Since singular points of $W$ are rational double points, 
its minimal nonsingular model $X$ is a K3 surface.
The proper transforms $E_{ij}$ of the lines $\ell_{ij} = \la p_i,p_j\ra$ and the proper transform $E_0$ of the twisted cubic $R_3$ 
is a set of 16 disjoint smooth rational curves 
on $X$. Let $E_i, i = 1,\ldots,6,$ be the exceptional curves over the nodes of $W$, and 
$\ell_{ijk}$ be the residual line in the intersection of $W$ with the plane $\Pi_{ijk} = \la p_i,p_j,p_k\ra$. 
The plane 
$\Pi_{lmn}$ with $\{i,j,k\}\cap \{l,m,n\} = \emptyset$ intersects $\Pi_{ijk}$ along a line $\ell$. 
It intersects $W$ at three points 
on lines $\ell_{ij}, \ell_{ik}, \ell_{jk}$ and $\ell_{lm}, \ell_{ln}, \ell_{mn}$. It follows that $\ell$ coincides with the line 
$\ell_{ijk}$. Thus we find another set of disjoint $(-2)$-curves $E_{ijk}$, the proper transforms of the lines 
$\ell_{ijk}$. It is immediate to check that the set 
$\calA$ of sixteen  $(-2)$-curves $E_0, E_{ij}$ and the set $\calB$ of $(-2)$-curves $E_i,E_{ijk}$ 
form an abstract symmetric configuration $(16_6)$ isomorphic to the Kummer configuration.
The surface $X$ has $20$ disjoint $(-2)$-curves; 16 of them come from the Kummer configuration, and four come from the 
resolution of $P$. 
 hence its Picard number is equal to $22$, and hence $X$ is a supersingular K3 surface.

\end{proof}

It was noticed by Hutchinson \cite{Hutchinson} that the Weddle surface admits a cubic Cremona involution: 
$$T: [x,y,z,w] \mapsto [a/x,b/y,c/z,d/w].$$
This works also in characteristic $2$, but instead of $8$ fixed points of $T$ outside $W$, there is a unique fixed point 
lying on $W$. This is the singular point $P$. Under the involution $T$ the curves $R_3$ and the line $\ell_{56}$ interchange.
This allows us to find the parametric form of $R_3$:
$$[s,t]\mapsto [\frac{a}{s+ta},\frac{b}{s+tb},\frac{c}{s+tc}, \frac{d}{s+td}].$$
We identify $T$ with its biregular lift to the nonsingular model $X$ of $W$. Then 
$T(E_1) = E_{234}, T(E_2) = E_{134}, T(E_3) = E_{124}, T(E_4)=E_{123}$. Any other line $\ell_{ijk}$ intersects 
two opposite edges of the coordinate tetrahedron, hence they form three orbits with respect to $T$.

\begin{lem}\label{pencil} The surface $X$ contains a quasi-elliptic pencil invariant with respect to $T$. It has three 
reducible fibers of type $\tilde{D}_4$ and 
eight reducible fibers of type $\tilde{A}_1^*$. The involution $T$ fixes one of the fibers of type $\tilde{D}_4$ and switches other 
fibers in pairs.  
\end{lem}

\begin{proof} Let $F = 2E_0+E_1+E_2+E_3+E_4$. It is immediate to check that $F$ is an effective nef divisor of arithmetic 
genus one and type $\tilde{D}_4$ (type $I_0^*$ in Kodaira's notation), hence $|F|$ is a genus one pencil on $X$. The image of $F$ under 
the involution $T$ is equal to  $G = 2E_{56}+E_{234}+E_{123}+E_{124}+E_{134}$. Since $G\cdot F = 0$, $G$ is another fiber 
of type $\tilde{D}_4$ and the pencil $|F|$ is $T$-invariant. The invariant fiber is the unique member of $|F|$ that contains the singular point $P$.
It is of type $\tilde{D}_4$. We also have six more reducible fibers $E_{ijk}+E_{ijk}'$, where $E_{ijk}$ is 
different from the components 
of $G$. Suppose that the pencil $|F|$ is an elliptic pencil. Adding up the Euler-Poincar\'e characteristics of 
the reducible fibers, we see that the sum is greater or equal than $3\times 6+2\times 6 = 30 > 24$. 
This contradiction shows that the
genus one pencil$|F|$  must 
be quasi-elliptic, also that there must be two more fibers of types $\tilde{A}_1^*$ (of type III in Kodaira's notation).
\end{proof} 

Since our family of supersingular K3 surfaces depends on three parameters (the projective equivalence classes of six points $p_1,\ldots,p_6)$, the Global Torelli Theorem for supersingular 
K3 surfaces suggests that the Artin invariant $\sigma_0$ of $X$ is equal to four. The previous lemma can be used to confirm this.

\begin{prop} The Artin invariant $\sigma_0$ of a general $X$ is equal to $4$.
\end{prop}

\begin{proof} Recall that the Artin invariant $\sigma_0$ of a supersingular K3 surface is equal to 
the half of the rank of the elementary abelian $2$-group equal to the discriminant group of the Picard lattice. 
We use the pencil $|F|$ from Lemma \ref{pencil}. The sublattice $M$ spanned by irreducible components of fibers modulo the divisor class of a fiber
is isomorphic to orthogonal sum of three copies of the root lattices $D_4$ and $8$ copies of the root latice of type 
$A_1$. The pencil has eight disjoint sections $E_{k5},E_{k6}, k = 1,2,3,4$, hence the Mordell-Weil group 
$\textrm{MW}(|F|)$ of the quasi-elliptic fibration $|F|$ contains a subgroup 
isomorphic to $(\bbZ/2\bbZ)^{\oplus 3}$. The Shioda-Tate discriminant formula 
$$|\mathrm{discr}(\mathrm{Pic}(X)|\cdot \#\textrm{MW}(|F|)^2 = |\mathrm{discr}(M)|$$
gives $\sigma_0\le 4$. Since our family is $3$-dimensional, we get  $\sigma_0\ge 4$. This proves the assertion.
\end{proof}

\begin{remark} It is known that the sum of $16$ disjoint $(-2)$-curves on a K3 surface in 
characteristic $p\ne 2$ is divisible by $2$, and the corresponding $\mu_2$-cover is birationally isomorphic 
to an abelian surface. I believe that in our case the sum $\Sigma$ of 16 disjoint curves $R,E_{ij}$ or $E_i,E_{ijk}$ is not 
divisible by $2$ if $X$ is a general member of the family. If it were, the Kummer $\mu_2$-cover of $X$ with the branch divisor 
$\Sigma$ is a non-normal surface (since $c_2(\Omega_X(\Sigma)) = -8$). It is known that every supersingular surface of Artin invariant 
$\sigma_0\le 3$ is birationally isomorphic to the quotient of the self-product of a rationa  
cuspidal curve by the infinitesimal group scheme $\boldsymbol{\mu}_2$ \cite{KondoSchroer}. The action depends on two parameters and the 
the quotient surface has 16 rational double points of type $A_1$ and one point of type $D_4$. I do not know whether 
the surfaces can be realized as degenerations of Weddle quartic surfaces.  
\end{remark}

 As in the case where $p\ne 2$, we can consider a rational map $\tau:\bbP^3\dasharrow L^* = \bbP^3$ given by the 
 web of quadrics $L$. The map factors through a separable Artin-Schreier cover of degree $2$, 
 the covering involution is defined 
 by an element of order $2$ in the normal subgroup $G\cong (\bbZ/2\bbZ)^5$ of the Weyl group $W(D_6)$ that leaves  
 invariant the projective orbit of the ordered point set $(p_1,\ldots,p_6)$ in the Coble-Cremona action of 
 $W(D_6)$ on the GIT-quotient $P_3^6$ of six ordered points in $\bbP^3$ (see \cite[\S 36]{Coble}, \cite{DO}).
It follows that the cover $\bbP^3$ is birationally isomorphic to a hypersurface
$$y_4^2+q_2(y_0,y_1,y_2,y_3)y_4+q_4(y_0,y_1,y_2,y_3) = 0$$
in the weighted projective space $\bbP(1,1,1,1,2)$, where 
$Q:= V(q_2)$ is the pfaffian hypersurface of the linear system $L$.
The Weddle surface is the ramification locus of this cover, and it is birationally isomorphic 
to an inseparable $\mu_2$-cover of the quadric $V(q_2)$ defined by the invertible sheaf $\calO_Q(B)$, where 
$B = V(q_2,q_4)$.

The rational map $\tau$ lifts to a separable regular degree 2 map
$\tilde{\tau}:\mathrm{Bl}_{p_1,\ldots,p_6}(\bbP^3)\to \bbP^3$. Its Stein factorization consists of a birational morphism 
$\mathrm{Bl}_{p_1,\ldots,p_6}(\bbP^3)\to Y$ that blows down the proper transforms of $R_3, \ell_{ij}$ and an Artin-Schreier 
finite map $f:Y\to \bbP^3$. The known formula for the canonical sheaf of an Artin-Schreier cover gives 
that the branch divisor of $f$ is the union of the quadric $Q$ and six planes, the images of the exceptional divisors 
of the blow-up. It follows that the branch divisor of the inseparable double cover $X\to Q$ is equal to the union 
of six conics, the images of the curves $E_1,\ldots,E_6$. They intersect at one point $x_0$, the image of $E_0$.

Let $\pi_{x_0}:Q\dasharrow \bbP^2$ be the projection of $Q$ to a plane with center at $x_0$. The projections of the conics 
are six lines $V(l_1),\ldots,V(l_6)$ in general linear position. The projections of conics $\tau(\ell_{ijk})$ are ten conics, each passing through 
six of the 15 intersection points $q_{ij} = l_i\cap l_j$. The differential $d\Phi$ of the curve $\Phi= V(l_1\cdots l_6)$ has $21$ 
zeros counted with multiplicities and each two conics intersect at two zeros outside the curve $\Phi$ \cite{Shimada}.
In fact, the ten conics intersect at two of these points, the projections of the lines on $Q$ passing through $x_0$.

The additional zero of $d\Phi$ comes with multiplicity four. It is equal to the projection of the image on $Q$ of 
the singular point of the Weddle surface $W$.

By choosing projective coordinates in the plane such that 
$$l_1 = x,\ l_2 = y, \ l_3 = z, \ l_4 = x+y+z,\  l_5 = a_1x+a_2y+a_3z,\ l_6 = b_1x+b_2y+b_3z,$$
one directly checks that the condition that there exists a conic through the intersection points 
$P_{ij} = V(l_i,l_j), 1\le i< j\le 3$ and $P_{mn} = V(l_m,l_n), 4\le m < n\le 6$ is equivalent to the condition that 
the six lines are dual to six points lying on a smooth conic $C$ in the dual plane. Also, 
an explicit computation shows that the line joining the two common points of the ten conics coincides with 
the line $\ell$ dual to the conic $C$. So, the double plane model of $X$:
$$w^2 + l_1\cdots l_6 = 0$$
 is an analog in characteristic $2$ of the double plane model of a Kummer quartic surface, where instead of the dual 
 line $\ell$ we 
have a contact conic to the lines.

 \begin{remark} In \cite{Shimada} I. Shimada gave a classification of supersingular K3 surfaces birationally isomorphic 
 to an 
 inseparable double plane with the branch curve $V(\Phi)$ of degree $6$. In this classification, it is assumed that the differential 
 $d\Phi$ has 21 simple zeros. Our surface does not appear in his list because $d\Phi$ has a multiple zero corresponding to the 
  singular point $P$ of $W$ of type $D_4$.
\end{remark}

\section{Six points in $ \bbP^1$}
We learned that a set of six points in $\bbP^1$ in arbitrary characteristic 
leads to a K3 surface of Kummer type of index 6 birationally isomorphic to a Weddle surface. We see the projective equivalence class 
of six points in different ways: the six points $p_1,\ldots,p_6$ on a twisted cubic $R_3$, the six points dual to the lines 
$V(l_i)$ in its double plane birational model, six intersection points of the lines lying on a conic, six intersection points 
of the lines $V(l_i)$ with the line $\ell$.
It is classically known that the  GIT-quotient $P_1^6:= (\bbP^1)^6/\!/\mathrm{PGL}_2(k)$ with respect to the 
democratic linearization 
is isomorphic to the Segre cubic primal $\Sigma_3$ representing the unique projective isomorphism class of a cubic 
hypersurface in $\bbP^4$ with $10$ ordinary nodes \cite[Chapter 1, \S 3]{DO}, \cite[Theorem 9.4.10]{CAG}. 
An equation of $\Sigma_3$ in all characteristics can be 
chosen to be the following:
\beq\label{segrec}
x_1x_2x_4-x_0x_3x_4-x_1x_2x_3+x_0x_1x_3+x_0x_2x_3-x_0^2x_3 = 0.
\eeq
The ten singular points of $\Sigma_3$ are: 
\begin{eqnarray}
\begin{split}
&[0,0,0,0,1],\ [0,0,0,1,0],\ [0,0,1,0,0],\ [0,1,0,0,0],\ [0,1,0,1,1],\\
&[0,0,1,1,1],\  [1,1,0,0,0],\  [1,0,1,0,0],\  [1,0,0,0,1],\ 
 [1,1,1,1,0].
 \end{split}
\end{eqnarray}
If $p\ne 2$, one can transform equation \eqref{segrec} to the familiar form
$$x_0^3+\cdots+x_4^3-(x_0+\ldots+x_4)^3 = 0,$$
which exhibits obvious $\frakS_6$-symmetry of the equation. 
 
 If $p = 2$, the symmetry is not obvious. Fix the set of reference points $p_1,\ldots,p_5$ in $\bbP^3$ from \eqref{six}. 
 An explicit rational parameterization 
 \beq\label{mapphi}
 \phi:\bbP^3\dasharrow \Sigma_3 \subset \bbP^4
 \eeq
 is given  by quadrics through the reference points:  
 $$[x_0,x_1,x_2,x_3,x_4] = [t_3(t_0+t_1),t_3(t_1+t_3),t_2(t_0+t_1),t_2(t_1+t_3),(t_0+t_2)(t_1+t_3)].$$
The action of the symmetric group $\frakS_6$ on $\Sigma_3$ is induced by its rational action on $\bbP^3$.
The transpositions 
$(01),(12),(2,3),(34)$ of $\frakS_6$ are given by transpositions $(01),(12),(23), (34)$ of 
coordinates $(t_0,t_1,t_2,t_3)$. The transposition $(45)$ acts by the standard Cremona involution 
$(t_0,t_1,t_2,t_3)\mapsto (\frac{1}{t_0},\frac{1}{t_1},\frac{1}{t_2},\frac{1}{t_3})$.
The action of $\frakS_6$ on $\Sigma_3$ is induced by a $5$-dimensional linear representation of $\frakS_6$: 
\begin{eqnarray*}
(12)&:&(x_0,\ldots,x_4)\mapsto (x_0,x_0+x_1,x_2,x_2+x_3,x_0+x_2+x_4),\\
(23)&:&(x_0,\ldots,x_4)\mapsto (x_0+x_1,x_1,x_0+x_1+x_4,x_1+x_3,x_0+x_2),\\
(34)&:&(x_0,\ldots,x_4)\mapsto (x_2,x_3,x_0,x_1,x_0+x_2+x_4),\\
(45)&:&(x_0,\ldots,x_4)\mapsto (x_0,x_1,x_0+x_2,x_1+x_3,x_0+x_1+x_4),\\
(56)&:&(x_0,\ldots,x_4)\mapsto (x_2,x_0+x_1+x_2+x_4,x_0,x_3+x_4,x_4).
\end{eqnarray*}
As in the case $p\ne 2$, the linear representation is an irreducible 
representation of $\frakS_6$ corresponding to the partition 
$\lambda = (3,3)$. 

 If $p\ne 2$, the dual hypersurface $\Sigma_3^*$ is  
a quartic hypersurface. The group $\frakS_6$ acts linearly 
in the dual projective space via its action on the partial derivatives of $\Sigma_3$. 
It defines an irreducible linear representation of $\frakS_6$ corresponding to the partition $(2,2,2)$.
In appropriate dual coordinates, $\Sigma_3^*$ can be given by the following equations in $\bbP^5$:
\beq\label{igusaq}
(\sum_{i=0}^5y_i^2)^2-4\sum_{i=0}^5y_i^4 = \sum_{i=0}^5y_i =0,
\eeq
which exhibit obvious $\frakS_6$-symmetry. The quartic hypersurface $\Sigma_3^*$ is isomorphic to the Igusa 
compactification of the 
moduli space $\mathcal{A}_2(2)$ of principally polarized abelian surfaces with a level two structure. 
For this reason, in modern literature, the quartic 
$\Sigma^*$ is called the \emph{Igusa quartic}, although, in classical literature, it was known as the \emph{Castelnuovo quartic}. For any 
smooth point $x\in \Sigma_3^*$, the tangent hyperplane at  $x$ cuts out $\Sigma_3^*$ along a quartic surface with 16 ordinary nodes. 
This is the Kummer surface of the Jacobian variety of the genus two curve associated to the corresponding point 
from $\Sigma_3$ \cite[p. 141]{Coble}. 

 The double cover of $\bbP^4$ branched along $\Sigma_3^*$ admits a modular interpretation as the GIT-quotient 
 $P_2^6: = (\bbP^2)^6/\!/\mathrm{PGL}_3(k)$ \cite{Coble}, \cite[Chapter 1]{DO}. Its equation in $\bbP(1,1,1,1,2]$ is 
 \beq\label{coblevar}
 w^2+F_4(y_0,\ldots,y_4) = 0,
 \eeq
 where we rewrite equations \eqref{igusaq} by eliminating $y_5$.\footnote{In modern literature, the  $4$-fold 
 given by \eqref{coblevar} is known as the Coble four-fold.} The involution $w\mapsto -w$ corresponds to the 
 association involution, and its locus of fixed points is the GIT-quotient of the subvariety of $(\bbP^2)^6$ of 
 ordered sets of points lying on a conic.
 
 If $p = 2$, the GIT-quotient $P_2^6$ is still defined. It is isomorphic to a hypersurface $\mathcal{V}$
 in $\bbP(1,1,1,1,1,2)$ given by equation:
\beq
w^2+w(y_2y_3+y_1y_4+y_0(y_0+y_1+y_2+y_3+y_4))+y_0y_1y_4(y_0+y_1+y_2+y_3+y_4) = 0.
\eeq
 The projection to the $y$-coordinates defines a separable double cover $\mathcal{V}\to \bbP^4$ branched over the quadric 
 $Q = V(q)$, where 
 $$q = y_2y_3+y_1y_4+y_0(y_0+y_1+y_2+y_3+y_4).$$
  The involution $w\mapsto w+q$ has the same geometric meaning as in 
 the case $p\ne 2$. The locus of fixed points $F$ of the involution is isomorphic to the inseparable double 
 cover of $Q\subset \bbP^4$. 
 As in the case $p \ne 2$,
 $F$ is singular over the pre-images of $15$-lines. They represent the closed semistable orbits of point sets 
 of the form $(a,a,b_1,b_2,b_3,b_4)$. Each line contains three points representing the closed semi-stable orbits 
 of point sets $(a,a,b,b,c,c)$. The incidence relation between the lines and the points is the famous Cremona--Richmond symmetric configuration 
 $(15_3)$.
 
The duality fails if $p = 2$: the Hessian of the cubic polynomial defining $\Sigma_3$ is identically zero. 
 The group $\frakS_6$ still acts on $\mathcal{V}$ via its action on $P_2^6$ defining a linear representation 
 in the dual space $\bbP^4$ corresponding to the partition $(2,2,2)$:
\begin{eqnarray*}
(12)&:&(y_0,\ldots,y_4)\mapsto (y_0,y_1,y_2,y_0+y_1+y_3,y_0+y_2+y_4),\\
(23)&:&(y_0,\ldots,y_4)\mapsto (y_0,y_3,y_4,y_1,y_2),\\
(34)&:&(y_0,\ldots,y_4)\mapsto (y_1,y_0,y_2,y_3,y_0+y_1+y_2+y_3+y_4),\\
(45)&:&(y_0,\ldots,y_4)\mapsto (y_0,y_2,y_1,y_4,y_3),\\
(56)&:&(y_0,\ldots,y_4)\mapsto (y_0,y_1,y_0+y_1+y_2,y_3,y_0+y_3+y_4),\\
\end{eqnarray*}
The quadric $Q$ is invariant with respect to the representation. Since the partition $(2,2,2)$ is not $2$-regular,
the linear representation is reducible \cite{James}. 
In fact, one observes that the vector $(1,1,1,1,1)$ is invariant.

 \begin{thm} Let $x\in \Sigma_3 \subset \bbP^4$ be a nonsingular point and $Q_x$ be the polar quadric of 
 $\Sigma_3$ with pole 
 at $x$. The pre-image $X$ of $Q_x$ under the map $\phi:\bbP^3\dasharrow \bbP^4$ is isomorphic to the 
 Weddle surface associated with $6$ points $(p_1,\ldots,p_5,p_6 = \phi^{-1}(x))$.
 \end{thm}
 
 \begin{proof} Since the map $\phi$ from \eqref{mapphi} is given by the web of quadrics $L$, the pre-image of $Q_x$ is a quartic surface $W$ in $\bbP^3$ with 
 double points at $p_1,\ldots,p_5$. Since $Q_x$ is tangent to $\Sigma_3$ at the point $x$, 
 the quartic acquires an additional double point 
 at $p_6 = \phi^{-1}(x)$. The images of the lines $\ell_i = \la p_i,p_6\ra, 
 i = 1,\ldots,5,$ are lines 
 on $\Sigma_3$ passing through $x$. It is known that the polar quadric $Q_x$ 
 intersects $\Sigma_3$ at points $y$ such that the tangent hyperplane of $\Sigma_3$ at $y$ 
 contains $x$ \cite[Theorem 1.1.5]{CAG}.
 This implies that the five lines $\phi(\ell_i)$ are contained in $Q_x$, and hence 
 the lines $\ell_i$ are contained in $W$. Since $Q_x$ passes through singular points of $\Sigma_3$,
 the lines $\la p_i,p_j\ra, 1\le i<j\le 5$ are also contained in $W$. 
 
 Let $R_3$ be the unique twisted cubic through the six points $p_1,\ldots,p_6$. Its image in $\Sigma_3$ is the 
 sixth line in $\Sigma_3$ passing through $x$. By above, it is also contained in $Q_x$, hence
 $W$ contains $R_3$. It follows from Section 2 that $W$ is a Weddle surface.
 \end{proof}

\section{Congruences of lines and quartic del Pezzo surfaces}
 
A congruence of lines in $\bbP^3$ is an irreducible surface $S$ in the Grassmannian $\bbG:= G_1(\bbP^3)$ of lines in $\bbP^3$.
A line $\ell_s$ in $\bbP^3$ corresponding to a point  $s\in S$ is called a ray of the congruence. 
The algebraic cycle class $[S]$ of $S$ in the Chow ring $A^*(\bbG)$ is determined by two numbers, the order $m$ and the class $n$.
The order $m$ (resp. the class $n$) is equal to the number of rays passing through a general point $x$ in $\bbP^3 $ 
(resp. contained in a general plane $\Pi\subset \bbP^3$). We have $[S] = m\sigma_x+n\sigma_\Pi$, where 
$\sigma_x$ (resp. $\sigma_\Pi$) is the algebraic cycle class of 
 an $\alpha$-plane $\Omega(x)$ of lines through a point $x\in \bbP^3$ (resp. 
of a $\beta$-plane $\Omega(\Pi)$ of 
lines contained in a plane $\Pi$). The degree of the surface $S$ in the 
Pl\"ucker embedding $\bbG\hookrightarrow \bbP^5$ is equal to $m+n$.

The universal family of rays $Z_S = \{(x,s)\in \bbP^3\times S:x\in \ell_s\}$ comes with two projections 
$p_S:Z_S\to \bbP^3$ and $q_S:Z_S\to S$. 

We assume that $m = n = 2$ and $S$ is smooth. Then $S$ is a quartic 
del Pezzo surface in its Pl\"ucker embedding that coincides with its anti-canonical embedding. It follows that 
$S$ is contained in a hyperplane section $H\cap \bbG$, a linear complex of lines. 

By the definition of the order of a congruence, the cover $p_S:Z_S\to \bbP^3$ is of degree $2$.  It is known 
that $S$ does not contain fundamental curves, i.e. curves in $\bbP^3$ over which the fibers are one-dimensional. Thus 
the cover $p_S$ is a finite cover over the complement of a finite set of points. 

Let us assume now that $p\ne 2$ and see later what happens in the case $p = 2$. 
Although the classical theory of congruences of lines
assumes that the ground field is the field of complex number, 
all the facts are true only assuming that $p$ does not divide the order and the 
class (see a brief exposition of the theory of congruences in \cite[\S 2]{DolgReid}). 
The cover $p_S:Z_S\to \bbP^3$ is 
a Kummer type double cover branched along the \emph{focal surface} $\Foc(S)$ of $S$. 
The focal surface is a quartic 
Kummer surface with 16 nodes. The congruence is one of the six irreducible components of order $2$ of the 
surface of bitangent lines to $\Foc(S)$. If the Pl\"ucker equation of $\bbG$ is taken to be  
$\sum_{i=1}^6x_i^2 = 0$, the equations of the six congruences of bitangents are $x_i = 0$.

 The pre-image of a ray $\ell_s$ under $p_S$ in $Z_S$ is equal to the union of the fiber 
$q_S^{-1}(s)$ (that can be identified with $\ell_s$) and a curve $L_s$ which is projected to 
$C(s) = S\cap \bbT_s(\bbG)$ under the map $q_S:Z_S\to S$. The intersection points $L_s\cap q_S^{-1}(s)$ are the 
pre-images of the tangency points of $\ell_s$ with $\Foc(S)$. The map $L_s\to C(s)$ is the
 normalization map, the points 
in $L_s\cap q_S^{-1}(s)$ correspond to the branches of $C(s)$ at the singular point $s\in C(s)$. The locus 
of the pairs of points $L_s\cap q_S^{-1}(s)$ defines a double cover $q_S':X\to S$ of the ramification divisor $X$ of $p_S$. 
The ramification curve $R$ of $q_S'$
is the locus of 
the pre-images in $Z_S$ of points in $\Foc(S)$, where a ray $\ell_s$ is tangent to $\Foc(S)$ with multiplicity $4$.
The branch curve $B$ of $q_S'$ is the locus of points $s\in S$ such that the curve $C(s)$ has a cusp at $s$. It is known that 
$B\in |-2K_S|$ \cite[(2.9)]{DolgReid}. The curve $B$ is cut out by a quadric in $\bbP^5$. 
The adjunction formula shows that $X$ is a
 K3 surface. 
 
 The first projection $p_S:X\to \Foc(S)$ is a minimal resolution of singularities. The fibers $E(x_i)$ over 
 the singular points $x_i$ of $\Phi(S)$ form a set $\calA$ of 16 disjoint $(-2)$-curves. Another set $\calB$ of 
 $16$ disjoint $(-2)$-curves is obtained as the intersection of the plane $\Pi(x_i)$ swept by the 
 rays from $\Omega(x_i)$ with $\Foc(S)$. The plane $\Pi(x_i)$ is tangent to $\Foc(S)$ 
 along a conic. In classical terminology, such a plane is a 
 \emph{trope} and the corresponding conic is a \emph{trope-conic}. The map $T(x_i)\to E(x_i)$ is defined by the deck transformation 
 of the cover $q_S:X\to S$. It follows  that each line on $S$ splits under the cover $q_S:X\to S$. This is a remarkable 
 property 
 of the curve $B$: it is a curve in $|-2K_S|$, which is tangent to all lines contained in $S$.

 Assume now that $p = 2$. We still have a 
 realization of a quartic del Pezzo surface $S$ as a congruence of lines in $\bbP^3$ of order 2. It is equal to a hyperplane 
 section of quadratic line complex $\frakC$ which we may assume to be smooth. It follows that the order and the class of $S$ is equal to $2$.

 By definition  
 of the order of a congruence, the projection $p_S:Z_S\to \bbP^3$ is a map of degree $2$. Its general fiber is equal 
 to the intersection of the smooth conic $\Omega(x)\cap \frakC$ with a hyperplane section of $\bbG$. 
 Since $S$ is smooth, it consists 
 of two points.  This shows that the map $p_S$ is separable.  Let $Z_S\to Z_S'\to \bbP^3$ be its Stein factorization, 
 where the first map is a birational 
 morphism and the second map is a separable finite morphism of degree $2$. 
 Since $H^1(\bbP^3,\calO_{\bbP^3}(n)) = 0, n \ge 0$, the cover $Z_S'\to \bbP^3$ is an 
 Artin-Schreier cover. The known formula for the canonical class of the universal family of lines $Z_\bbG$ over 
 $\bbG$ \cite[10.1.1]{CAG} easily gives that 
 $\omega_{Z_S}\cong p_S^*\calO_{\bbP^3}(-2)\otimes q_S^*\omega_S(1)$. The formula for the canonical sheaf of 
 an Artin-Schreier cover shows that $Z_S'$ is given by an equation:
 \beq\label{kummer}
 x_4^2+F_2(x_0,x_1,x_2,x_3)x_4+F_4(x_0,x_1,x_2,x_3) = 0,
 \eeq
 where $Q = V(F_2)$ is a quadric and $V(F_4)$ is a quartic surface. The quartic polynomial $F_4$ is defined up to 
 a replacement $F_4$ with $A^2+AF_2+F_4$,
where $A$ is a quadratic form.

\begin{remark} As comminicated to me by T. Katsura, one can give an explicit equation of the quadric $V(F_2)$ in terms of 
the equation of a congruence of lines $S$ of order $2$ and class $2$. If 
$\bbG_1(\bbP^3) =  V(x_1y_1+x_2y_2+x_3y_3)\subset \bbP^5$ and $S$ is given by equations
\begin{eqnarray*}
a_1x_1y_1+a_2x_2y_2+a_3x_3y_3+c_1y_1^2+c_2y_2^2+c_3y_3^2 = 0,\\
\alpha_1x_1+\alpha_2x_2+\alpha_3x_3+\beta_1y_1+
\beta_2y_2+\beta_3y_3 = 0.
\end{eqnarray*}
then 
\beq
\begin{split}
&F_2 = (a_1+a_3)(\alpha_2x_0x_2+\beta_2x_1x_3)+(a_2+a_3)(\alpha_1x_0x_1+\beta_1x_2x_3)\\
&+
(a_1+a_2)(\alpha_3x_0x_3+\beta_3x_1x_2).
\end{split}
\eeq
\end{remark}

\vskip4pt
The following proposition is an analog  of the 
 description of $S$ as an irreducible component of the surface of bitangent lines to $\Foc(S)$. 
 
 \begin{prop} The congruence $S$ is an irreducible component of the locus of  
 points in $\bbG$ parametrizing lines in $\bbP^3$ that split under the cover $p_S:Z_S\to \bbP^3$ into two 
 irreducible components. In particular, no ray of the congruence is contained in the quadric $Q = V(F_2)$. 
\end{prop}

\begin{proof} The fiber $q_S^{-1}(s)$ maps isomorphically to the ray $\ell_s$ under the projection 
$p_S:Z_S\to \bbP^3$.
Thus the pre-image $p_S^{-1}(\ell_s)$ is equal to the union of $q_S^{-1}(s)$ and a curve $L_s$ whose points 
are the pre-images 
of rays intersecting $\ell_s$. The image of $L_s$ in $S$ under the projection $q_S:Z_S\to S$ is equal to the hyperplane section $C(s):=\bbT_s\cap S$.
If $\ell_s\subset Q$, then the restriction of $p_S$ over $\ell_s$ is a purely inseparable cover, so the pre-image of $\ell_s$ does not split.
\end{proof}

\vskip4pt
Note the last assertion is an analog of the fact that $\Foc(S)$ does not contain lines.

\vskip4pt
 I do not know how to describe explicitly the locus of splitting lines 
under a separable double cover. However, the condition for splitting of a line is clear. A separable cover 
$y^2+a_k(t_0,t_1)y+b_{2k}(t_0,t_1) = 0$ of a line with coordinates $t_0,t_1$ is reducible if and only if 
$b_{2k} = a_kc_k+c_k^2$ for some binary form $c_k$ of degree $k$. 

The pre-image of a general plane $\Pi$ under the map $p_S$ in $Z_S$ is a separable double cover given by equation 
$$w^2+a_2(t_0,t_1,t_2)w+b_4(t_0,t_1,t_2) = 0,$$
where $(t_0,t_1,t_2)$ are coordinates in $\Pi$. The double cover is isomorphic to a del Pezzo surface of degree $2$. It is known that it has 
$28$ 
lines which are split under the cover. They correspond to $56$ $(-1)$-curves on the del Pezzo surface. 
The splitting lines are discussed 
in \cite{DM}, where they are called \emph{fake bitangent lines}. The variety of splitting lines is a 
congruence in $\bbG$ of class equal to $28$. It is an analog in characteristic $2$ of the congruence of bitangents 
of a Kummer surface.  Its order is known to be equal to $12$, and its class is equal to $28$. Also is known that 
the congruence of bitangents of a Kummer surface consists 
of six irreducible congruences of order $2$ and class $2$ and $16$ $\beta$-planes $\Omega(T)$, where $T$ is a trope. 
It is natural to conjecture that 
the congruence of splitting lines is also of degree 12 and $S$ is one of its six irreducible 
components of order $2$.

A general ray $\ell_s$ intersects the quadric $Q$ at two points, the pre-images of these two points in $Z_S$ correspond to 
the branches of the singular point $s\in C(s)$. This shows that the double cover $q_S:X\to S$ parameterizing the branches of the curves $C(s)$ is a separable Artin-Schreier cover 
of degree $2$.  In the blow-up plane model of $S$, $X$ is isomorphic to a surface of degree $6$ in the weighted projective space 
$\bbP(1,1,1,3)$ given by equation:
 \beq\label{kummer2}
 x_3^2+F_3(x_0,x_1,x_2)x_3+F_6(x_0,x_1,x_2) = 0.
 \eeq
 By the adjunction formula, $\omega_X\cong \calO_X$. A ray $\ell_s$ defines a cusp of $C(s)$ at $s$ if and only if it is tangent to the quadric $Q$. It is known that the lines in $\bbP^3$ 
 tangent to a smooth quadric surface are parametrized by the tangential quadratic line complex $\calT(Q)$. 
 It is singular along the locus of lines contained in $Q$ \cite[Proposition 10.3.23]{CAG}. Since 
 $\calT(Q)\cap S\in |-2K_S|$, we see that $\calT(Q)$ is tangent to $S$ along the curve $B= V(F_3)\in |-K_S|$.
 This differs from the case $p\ne 2$, where the branch curve $B$ belongs to $|-2K_S|$.
 
 Let $L_i$ be one of sixteen lines on the del Pezzo surface $S\subset \bbG$. 
 A line in $\bbG$ is a pencil of rays contained in a plane, i.e. $L_i = \Omega(x_i)\cap \Omega(\Pi_i)$ for
  some $x_i$ in a plane $\Pi_i$. All rays 
 $\ell_s, s\in L_i,$ 
 pass through $x_i$, hence the fiber $E(x_i)$ of $Z_S\to \bbP^3$ over $x_i$ is equal to the fiber 
 of the projections $p_S:X\to \bbP^3$. This implies that 
$X$ is singular over the point $x_i$. So, we have $16$ points 
 $x_i\in Q$, over which the map $p_S$ is not a finite morphism. The points are analogs of singular points of $\Foc(S)$. 
 We have also $16$ planes $\Pi(x_i)$, they are swept by the rays $\ell_s,s\in L_i$. Each plane $\Pi(x_i)$   
 intersects $Q$ along a conic $T(x_i)$. They are characteristic two analogs 
  of trope-conics of $\Foc(S)$. Both curves $E(x_i)$ and the proper transforms of $T(x_i)$'s 
  in $X$ are mapped to the 
  line $L_i$, so the line $L_i$ splits under the separable cover $q_S:X\to S$. This defines two sets 
  $\calA$ (of curves $E(x_i)$) and $\calB$ (of curves $T(x_i)$) of disjoint $(-2)$-curves on $X$.
  
  \begin{thm} Let $\tilde{X}$ be a minimal resolution of singularities of $X$. Then $\tilde{X}$ is a K3 surface 
  of Kummer type of index $6$ birationally isomorphic to a Weddle quartic surface.
  \end{thm} 
  
  \begin{proof} The known formula for the canonical class of a separable double cover $X\to S$ gives 
  $\omega_X\cong q^*(\omega_S(-K_S)) \cong \calO_X$. Let us look at the singularities of $X$. 
   
   The surface $X$ is an inseparable Kummer cover of the quadric $Q$ defined by a section of 
  $\calL^{\otimes 2}$, where $\calL \cong \calO_{Q}(2)$.  It is known 
  that its set of singular points is equal to the support of the scheme of zeros of a section $\nabla$ of 
  $\Omega_{Q}^1\otimes \calL^{\otimes 2}$. This follows from local computations: locally, the cover is given by 
  $z^2+f(x,y) = 0$, and local sections $\nabla = df(x,y)$ glue to a global section of $\Omega_{Q}^1\otimes \calL^{\otimes 2}$.
 An ordinary node is locally given by equation $z^2+xy = 0$, hence the local $\nabla$ is equal to n
  $ydx+xdy$ and it has a simple zero at the node. Any other singular point contributes more than one to $c_2$. 
  
  We have 
  $$c_2(\Omega_{Q}^1\otimes \calL^{\otimes 2}) = 
  c_2(\Omega_{Q}^1)+c_1(\Omega_{Q}^1)c_1(\calL^2)+c_1(\calL^2)^2 = 20.$$
 We skip the proof that $X$ is normal (the direct proof is rather elaborate but the fact follows 
 from the argument in the last paragraph of the proof). 
  It has sixteen singular points over the sixteen points $x_i$. So $X$ must have other singular points.
  It is known that the automorphism group $\mathrm{Aut}(S)$ contains a subgroup isomorphic to 
  $(\bbZ/2\bbZ)^{\oplus 4}$, 
  and that group 
  acts transitively on the set of 16 lines \cite[Theorem 3.1]{DD}. This implies that the sixteen points are ordinary nodes.  
 Since the total sum of the multiplicities of the remaining zeros of $\nabla$ is equal to $4$, $X$ has only  rational double points.
  This proves that a minimal resolution $\tilde{X}$ of $X$ is a K3 surface. 
    
   The sum of sixteen lines on $S$ is a divisor in the linear system 
  $|-4K_S|$.  The images on $S$ of extra singular points of $X$ lie outside the union of $16$ lines. 
  Thus the self-intersection of its pre-image on $\tilde{X}$ is equal to $128$. If $n$ is the index of the configuration 
  $\calA,\calB$ of $(-2)$-curves, then this self-intersection must be equal to $-64+32n$. This implies that $n = 6$.

 We have a Kummer configuration of 16 points and 16 conics on the quadric $Q$. Projecting $Q$ from one of the points, 
 we get six lines $V(l_i)$, the projections of the six conics containing the center of the projection map. 
 The projections of other ten conics are conics in the plane passing through six intersection points of the lines. 
 So, the surface $X$ is birationally isomorphic to the double plane $V(w^2+l_1\cdots l_6)$ as in the case of a 
 Weddle surface. The ten conics intersect at the projections of the rulings of $Q$ containing the center of the 
 projection map. We also get, as a bonus, that $X$ contains a rational double point of type $D_4$ (I believe that the image of the singular point under 
 the projection $q_S$ is the strange point on the del Pezzo surface $S$ defined in \cite{DD}).

\end{proof}

  
 \section{Rosenhain and G\"opel tetrads} A Rosenhain tetrad of a quartic Kummer surface is a subset of four nodes such 
 that the planes containing three of the nodes are tropes \cite[\S 50]{Hudson}. If one equips the set of $2$-torsion points 
 of $\Jac(C)$ with a structure of a symplectic four-dimensional 
 linear space over $\bbF_2$, then a Rosenhein tetrad is the image of a translate of a non-isotropic plane. There are $80$ Rosenhain tetrads. 
 Each Rosenhain tetrad defines a symmetric configuration $(4_3)$ between the sets of tropes and nodes. 
 The union of two Rosenhain tetrads 
 without common points form a symmetric configuration 
 $(8_4)$. This configuration is realized by $8$ vertices of a cube and $8$ faces of 
 two tetrahedra inscribed in the cube.
  
The union of two Rosemhain tetrads can be illustrated by the following figure (see \cite[7.3]{Manin}): 
 \begin{eqnarray}\label{figure}
 \begin{matrix}\circ&\star&&|&&\circ&\star\\
 \circ&\star&&|&&\circ&\star\\
\circ&\star&&|&&\circ&\star\\
 \circ&\star&&|&&\circ&\star\\
 \end{matrix}
 \end{eqnarray}
Here  circles correspond to the nodes and the stars correspond to tropes. Each side of the diagram 
represents a Rosenhain tetrad. A point in a row $i$ lies in the 
plane  
in the same row on the other side of the diagram, and it also lies in the three planes on the same side of the diagram 
from different rows.

Let us see how to get this configuration with the absence of the Kummer surface.
A quartic del Pezzo surface $S$ contains 20 pairs of tetrads of disjoint lines with the intersection relation
of each pair forming an abstract configuration 
$(4_3)$. We use a birational model of $S$ as the blow-up of 
five points $q_1,\ldots,q_5$ in 
the plane, and denote by 
$L_i$ the lines on $S$ coming from the exceptional curves over the points $p_i$,
10 lines $L_{ij}$ coming from the lines $\la q_i,q_j\ra$, and one line $L_0$ coming 
from the conic through the five points. Then the 20 pairs are the following:
\begin{itemize}  
\item $10$ pairs
$$\{L_0,L_{ij},L_{ik},L_{jk}\}, \quad \{L_i,L_j,L_k, L_{lm}\},$$
\item 10 pairs 
$$\{L_i, L_{ij},L_{ik},L_{il}\}, \quad \{L_m,L_{jm},L_{km},L_{lm}\}.$$
\end{itemize}
  
 Each tetrad of lines on $S$ from above splits in $X$ into $8$ disjoint $(-2)$-curves.
The curves correspond 
to the first two columns in the diagram, the other tetrad in the pair defines the third and the fourth columns.
We get only $40$ Rosenhain tetrads in this way. If $p\ne 2$, other $40$ Rosenhain tetrads arise from different 
congruences of lines, which define different irreducible components of the surface of bitangent lines of the Kummer surface.
I believe that the same is true if $p = 2$: other $40$ tetrads arise from different components of the surface of splitting lines 
of the double cover \eqref{kummer}. 
\vskip4pt
Note that a configuration of type $(4_3)$ is realized by two sets of lines among 20 lines on an octic model of the Kummer 
surface in characteristic two \cite[Figure 2]{KatsuraKondo}.

\vskip4pt
A \emph{G\"opel tetrad} is a subset of four nodes such that no three of them lie on a trope. 
There are 60 G\"opel tetrads. They correspond to the translated isotropic planes in $\bbF_2^4$. 
To get them from a quartic del Pezzo surface 
$S$, one considers 30 subsets of four skew lines $(L_i,L_j,L_{kl},L_{km})$, where $\{i,j\}\cap \{k,l,m\} = \emptyset$.
The pre-image of each subset in $X$ defines a set of four tropes and four 
points $x_i$ forming a G\"operl hexad. There will be $30$ G\"opel tetrads arising in this way. As in the case of Rosenhain 
tetrads, other $30$ G\"ope tetrads should arise from other irreducible components of the surface of splitting lines of 
\eqref{kummer}.

Recall from \cite[Theorem 1.20]{Gonzalez} that there are three abstract configurations of type $(16_6)$. The Kummer one is 
non-degenerate in the sense that any pair of trope-conics have two common vertices. It follows from 
our construction of Kummer configurations that they are non-degenerate. If $p\ne 2$, any non-degenerate Kummer configuration 
of points and planes of type $(16_6)$ is realized on a Kummer quartic surface. As we see, in 
characteristic $2$ this is not true anymore, and the Kummer surface should be replaced by a quadric surface.

\end{document}